%% file: main.tex
\documentclass{article}
\usepackage[utf8]{inputenc}
\usepackage[T2A]{fontenc}
\usepackage[english]{babel}

\usepackage{todonotes}

\usepackage[a4paper, margin=1.5cm]{geometry}

\usepackage{amsmath,amssymb,amsthm,amsfonts}
\usepackage{mathrsfs}
\usepackage{mathtools}
\mathtoolsset{showonlyrefs}
\usepackage{hyperref}
\hypersetup{
    colorlinks=true,
    linkcolor=red,
    filecolor=magenta,      
    urlcolor=cyan,
}

\usepackage{microtype}

\newcommand{\ex}{\mathrm{ex}}

\renewcommand{\phi}{\varphi}
\newcommand{\eps}{\varepsilon}

\newtheorem{lemma}{Lemma}
\newtheorem{theorem}{Theorem}

\newtheorem{proposition}{Proposition}
\newtheorem{conjecture}{Conjecture}
\newtheorem{problem}{Problem}
\newtheorem{definition}{Definition}

\newtheorem{corollary}{Corollary}

\title{Erd{\H o}s--Hajnal problem for $H$‐free hypergraphs}
\author{Danila Cherkashin, Alexey Gordeev\thanks{The work was supported by the Foundation for the Advancement of Theoretical
Physics and Mathematics ``BASIS''.}, Georgii Strukov\footnotemark[1]}

\begin{document}

\maketitle

\begin{abstract}
    This paper deals with the minimum number $m_H(r)$ of edges in an $H$-free hypergraph with the chromatic number more than $r$. We show how bounds on Ramsey and Tur\'an numbers imply bounds on $m_H(r)$.
\end{abstract}

\input{results/intro.tex}
\input{results/polyRamsey}
\input{results/tournament_K34.tex}
\input{results/pluhar.tex}

\input{results/rb.tex}

\input{results/diss.tex}

\bibliography{main}
\bibliographystyle{plain}

\paragraph{Danila Cherkashin}
Institute of Mathematics and Informatics,
Bulgarian Academy of Sciences\\
Sofia 1113, 8 Acad. G. Bonchev str.\\
jiocb.orlangyr@gmail.com

\paragraph{Alexey Gordeev and Georgii Strukov}
The Euler International Mathematical Institute, St. Petersburg, Russia

\end{document}

%% file: results/intro.tex

\section{Introduction}

A \textit{hypergraph} $G = (V, E)$ consists of a finite set of vertices $V$ and a family $E$ of subsets of $V$, which are called edges.
A hypergraph $G$ is called \textit{$n$-graph} if every edge in $G$ has size $n$.
A \textit{vertex $r$-coloring} of a hypergraph $G = (V, E)$ is a map from $V$ to $\{1, \dots , r\}$.
A coloring is \textit{proper} if there are no monochromatic edges, i.e. any edge $g \in E$ contains two vertices of different colors.
The \textit{chromatic number} of a hypergraph $G$ is the smallest number $\chi (G)$ such that there exists a proper $\chi (G)$-coloring of $G$. 
A vertex subset is \textit{independent} in $G$ if it does not contain an edge of $G$.
The \textit{independence number} of $G$ is the maximal size $\alpha(G)$ of an independent set.

Although the subjects are defined and studied for $n$-graphs, this paper deals with 3-graphs only.
For convenience, we restrict cited results on 3-graphs despite most of them being more general.

Let us start with best known bounds for 3-graphs without restrictions.
Let $m(r)$ be the smallest number of edges in a 3-graph with the chromatic number at least $r + 1$. Then  
\[
(1+o(1)) \frac{4}{e^2} r^3 \leq m(r) \leq \binom{2r+1}{3} = (1+o(1)) \frac{4}{3}r^3,
\]
see~\cite{cherkashin2019Erdos} for details. See also the survey on related problems on $n$-graphs~\cite{raigorodskii2020extremal}.

A subgraph of a 3-graph $G = (V,E)$ is a 3-graph $G' = (V',E')$, where $V' \subset V$, $E' \subset E$ and $E' \subset \binom{V'}{3}$. 
An \textit{induced subgraph} of a 3-graph $G = (V,E)$ is a 3-graph $G[V']$, formed from a vertex set $V' \subset V$ and \textbf{all} of the edges from $G$ contained in $V'$. An $H$-free 3-graph is a 3-graph that does not contain $H$ as a subgraph.

Bohman, Frieze and Mubayi~\cite{bohman2010coloring} introduced the quantity $m_H(r)$, which is the minimal number of edges in an $H$-free 3-graph with chromatic number more than $r$.
First of all, note that sometimes $m_H(r)$ is undefined. More precisely, this is the case iff $H$ is a \textit{hyperforest}, i.e. every two edges of $H$ share at most one vertex and $H$ has no cycles.
We discuss it with more details later.

\paragraph{Structure of the paper.} The rest of the introduction is devoted to the known results, our contribution and classical tools.
Subsection~\ref{subsec:known} contains known bounds on $m_H(r)$, mostly from paper~\cite{bohman2010coloring}.
Our results are listed in Subsection~\ref{subsec:contr}, where we show some relations $m_H(r)$ has with Ramsey and Tur{\'a}n numbers.
The tools which we use in proofs of these relations are recalled in Subsection~\ref{subsec:tools}.
The correctness of the definition of $m_H(r)$ is discussed in Subsection~\ref{subsec:corr}.
In Section~\ref{proofs} we give proofs of the results. In Section~\ref{diss} we list some open questions.

\subsection{Known bounds}
\label{subsec:known}

\begin{definition}
Let $H$ be a 3-graph. Define the \textit{balance} $\rho(H)$ of $H$ as follows
\[
\rho (H) := \max_{H'} \frac{e' - 1}{v' - 3},
\]
where $H'$ is a (nonempty) subgraph of $H$ with $v'$ vertices and $e'$ edges.
We say that $H$ is \textit{balanced} if this maximum occurs for $H$.
\end{definition}

\begin{theorem}[\cite{bohman2010coloring}]
\label{thm:bohman2010-4}
If $H$ is a balanced 3-graph with $\rho(H) > \frac{1}{2}$, then
\[
m_H(r) = O \left ( (r^2 \log r)^{\frac{3\rho-1}{2\rho-1}} \right ).
\]
\end{theorem}

In the following case an $H$-free 3-graph is also called a simple or linear 3-graph.
\begin{theorem}[\cite{kostochka2001chromatic,bohman2010coloring}]
\label{thm:kostochka-bohman}
Let $H$ be the 3-graph on 4 vertices with the set of edges $\{ \{1,2,3\},\{1,2,4\} \}$, then
\[
m_H(r) = \theta \left ( r^4 \log^2 r \right ).
\]
\end{theorem}

In the following case an $H$-free 3-graph is also known as a graph with independent neighbourhoods. The upper bound is a corollary of Theorem~\ref{thm:bohman2010-4}.
\begin{theorem}[\cite{bohman2010coloring}]
Let $H$ be the 3-graph on 5 vertices with the set of edges  $\{ \{1,2,3\},\{1,2,4\}, \{1,2,5\}, \{3,4,5\} \}$, then
\[
c r^{\frac{7}{2}} \leq m_H(r) \leq C r^{\frac{7}{2}} \log^\frac{7}{4}r.
\]
\end{theorem}

Hypergraph $H$ in the following theorem is also known as $K^3_4$; further we denote it by $K_4$ since we have a deal with 3-graphs only.

\begin{theorem}[\cite{bohman2010coloring}]
Let $H$ be the 3-graph on 4 vertices with the set of edges $\{ \{1,2,3\},\{1,2,4\}, \{1,3,4\}, \{2,3,4\} \}$, then
\[
m_H(r) = O \left( r^3 \log^3 r \right).
\]
\end{theorem}

\subsection{Results} 
\label{subsec:contr}

\subsubsection{Relation with Ramsey numbers}

A \textit{Ramsey number} $R(H, K_t)$ is the minimal number of vertices in a 3-graph $G$ that guarantees existence of a copy of $H$ or an independent set on $t$ elements in $G$.
We show that bounds on Ramsey numbers imply bounds on $m_H(r)$.

\begin{theorem}\label{thm:rPoly}
    Consider a 3-graph $H$. Let $a, b > 1$ and $A, B > 0$ be numbers such that for all $t \geq 1$ 
    \begin{equation}\label{eq:ramsey_bounds}
        At^{a} < R(H,K_t) \leq Bt^{b}.
    \end{equation}
    Then for all $r \geq 1$
    \begin{equation}
        B'r^{\frac{b}{b-1} + 2} \leq
            m_H(r) 
        \leq A'r^{\frac{3a}{a-1}},
    \end{equation}
    where 
    \begin{equation}
        A' = \frac{(1+o(1))^{3a}}{6A^{\frac 3{a-1}}}, \qquad
        B' = \left(
            \frac{(1-2^{1-b})}
                {2^{\frac{b}{b-1}+b+3}(12e)^{b-1} B \bigl(\frac b{b-1}+2\bigr)^b}
        \right)^{1 / (b - 1)}.
    \end{equation}
\end{theorem}

In~\cite{kostochka2013hypergraph} it was shown that for all $t \geq 1$
\begin{equation}
    \frac{at^{3/2}}{(\log t)^{3/4}} \leq
    R(C_3,K_t) \leq
    bt^{3/2},
\end{equation}
where $C_3$ is a \textit{loose cycle} with 3 edges, i.e. the 3-graph on 6 vertices with the set of edges $\{\{1, 2, 3\}, \{3, 4, 5\}, \{5, 6, 1\}\}$, and $a, b > 0$ are some constants.
A straightforward application of Theorem~\ref{thm:rPoly} gives
\begin{corollary}
    There is a constant $C > 0$ such that for all $r \geq 1$ 
    \begin{equation}
        Cr^{5} 
        \leq m_{C_3}(r). 
    \end{equation}
\end{corollary}
However, the upper bound can be made more precise than the one implied by Theorem~\ref{thm:rPoly}.
\begin{proposition}\label{prop:good_upper_bound}
    There is a constant $C > 0$ such that for all $r \geq 1$
    \begin{equation}
        m_{C_3}(r) 
        \leq Cr^{5}(\ln r)^{5/2}.
    \end{equation}
\end{proposition}

\begin{theorem}\label{thm:rExp}
    Suppose that for a 3-graph $H$ and some constants $A, B > 0$ and $b > 1$ we have
    \begin{equation}\label{eq:ramsey_exp_b}
        t^{At} < R(H,K_t) \leq t^{Bt^b}.
    \end{equation}
    Then 
    \begin{equation}
        r^3 \ln^{1/b-o(1)} r
        \leq m_H(r) \leq (1+o(1)) \frac{8}{A^3} r^3  \left(\frac{\ln r}{\ln \ln r}\right)^{3}.
    \end{equation}
 \end{theorem}


Let $K_4^{-}$ be the 3-graph with four vertices and three edges.
In~\cite{fox2019independent} it was stated that
\begin{equation}
    R(K_4^-, K_t) = t^{\Theta(t)}.
\end{equation}  
Then Theorem~\ref{thm:rExp} immediately yields
\begin{corollary}
    For some constant $C > 0$,
    \begin{equation}
        r^3 \ln^{1-o(1)} r
        \leq m_{K_4^-}(r) \leq C  r^3  \left(\frac{\ln r}{\ln \ln r}\right)^3.
    \end{equation}
\end{corollary}
The following bounds are from~\cite{conlon2010hypergraph}:
\begin{equation}
    2^{c_1 n\log n} \leq R (K_4,K_n) \leq 2^{c_2n^2\log n}
\end{equation}
for some constants $c_1, c_2 > 0$ and any sufficiently large $n$.
A direct application of Theorem~\ref{thm:rExp} gives
\begin{corollary}
    For some constant $C > 0$,
    \begin{equation}
        r^3 \ln^{1/2-o(1)} r
        \leq m_{K_4}(r) \leq C  r^3  \left(\frac{\ln r}{\ln \ln r}\right)^3.
    \end{equation}    
\end{corollary}

Moreover, using the same idea of graph reduction by coloring some big independent set in the same color, as in proof of theorems of this section (but with much less effort), we obtain
\begin{proposition}\label{prop:e288}
    Let $H$ be the hypergraph on 7 vertices with the set of edges $\{\{1,2,3\},\{1,4,5\},\{1,6,7\},\{2,4,6\} \}$. Then
    \begin{equation}
        m_H(r) > Cr^4,
    \end{equation}
    where $C = e/288$.
\end{proposition}

\subsubsection{Relation with Tur{\'a}n numbers}

A \textit{Tur\'an number} $\ex(n,H)$ is the maximum number of edges in an $H$-free 3-graph with $n$ vertices. As with Ramsey numbers, we can obtain bounds for $m_H(r)$ from bounds on Tur\'an numbers.

\begin{theorem}\label{thm:turan}
    Let $H$ be a 3-graph and let $C_1 > 0$ be a constant such that for all $n \geq 1$
    \begin{equation}\label{eq:exnH}
        \ex(n,H) \leq C_1n^b.
    \end{equation}
    Then for any constant $C_2 > 0$ there exists some $R = R(C_2)$ such that for any $r > R$
    \begin{equation}\label{eq:r4log}
        m_H(r) > \frac{C_2r^\frac{2b}{b-1}}{(\ln \ln r)^{1+\frac{2b}{b-1}}}.
    \end{equation}
\end{theorem}

The following proposition is folklore.

\begin{proposition}\label{prop:edges_ordering}
Let $H$ be a 3-graph, such that $t=|V(H)|=|E(H)|+2$. Suppose that there is an ordering of edges  $h_1,\dots,h_{t-2}$ such that $h_i=\{a_i,b_i,c_i\}$ for each $i>1$, where $\{a_i,b_i\}\subset h_{j(i)}$ for some ${j(i)}<i$ and $c_i\not\in h_j$ for $j<i$. Then 
\begin{equation}    
    ex(n,H)\leqslant (t-3)\binom{n}{2}.
\end{equation}
\end{proposition}

Note that such a graph $H$ is balanced with balance $1$, so Theorems~\ref{thm:turan} and~\ref{thm:bohman2010-4} give
\begin{corollary}
    Suppose that 3-graph $H$ satisfies the conditions of Proposition~\ref{prop:edges_ordering}. Then for some positive constants $C_1, C_2 > 0$ and for any sufficiently large $r$ 
    \begin{equation}
        C_1 \frac{r^4}{(\ln \ln r)^5} < m_H(r) < C_2 r^4 (\ln r)^2.
    \end{equation}
\end{corollary}

\subsection{Tools}
\label{subsec:tools}

\subsubsection{Applications of LLL}

We are going to use the following classical technique.
\begin{lemma}[Symmetric Lov\'asz local lemma]\label{lm:LLL}
    Let $A_1,\ldots,A_n$ be events with probabilities
    $\mathbf{P}(A_i) \leq p$ for $1 \leq i \leq n$.
    Suppose that each event $A_i$ is mutually independent of all the other events except at most $d$ of them. 
    If $ep(d+1) < 1$, then $\mathbf{P}(\wedge_{i=1}^n \overline{A}_i) > 0$.
\end{lemma}
The following lemma is a folklore corollary of Lemma~\ref{lm:LLL}.
\begin{lemma}\label{lm:3graph-coloring}
    Let $G$ be a 3-graph. Let $r>0$ be some integer. Suppose that for every $v\in V(G)$
    \begin{equation}
        \deg v \leq \Delta  = \frac{r^2}{3e}.
    \end{equation}
    Then $G$ can be properly colored with 
    $r$
    colors.
\end{lemma}

\begin{proof}
    Let us assign to each vertex one of $r$ colors at random uniformly and independently.
    For each $f \in E$ consider an event $A_f = \text{``$f\in E$ is monocromatic''}$. The probability of such event is $1/r^2 =: p$. Clearly each event $A_f$ is independent of $\{A_{g} \colon f \cap g = \emptyset\}$. There are at most $3(\Delta-1) + 1 < 3\Delta$ events (including $A_f$ itself) that are not in this family, so
    \begin{equation}
       ep(d+1) <  \frac{3e\Delta}{r^2} = 1,
    \end{equation}
    and, by Lemma~\ref{lm:LLL}, $G$ has some coloring in $r$ colors with no monochromatic edges.
\end{proof}

Lemma~\ref{lm:3graph-coloring} allows us to care only about vertices with big degree when we are coloring some graph in $r$ colors. Let us make this statement more precise. 
\begin{definition}
    Denote $c = (12e)^{-1}$.
    Consider an integer $r$ and a 3-graph $G$.
    Define
    \begin{equation}
        V_{small} = \{v \in V(G) \mid \deg v \leq cr^2\},
    \end{equation}
    \begin{equation}
        V_{big} = V(G) \setminus V_{small}.
    \end{equation}
\end{definition}
By a direct application of Lemma~\ref{lm:3graph-coloring} we obtain
\begin{lemma}\label{lm:vsmall}
    Let $r$ be an integer and $G$ be a 3-graph.
    Then $G[V_{small}]$ is $\lceil r/2\rceil$-colorable.
\end{lemma}

\subsubsection{Greedy coloring}

The following approach was introduced by Pluh{\'a}r~\cite{Pl}. Let us consider an arbitrary linear order $\pi$ on the vertex set of a given 3-graph $G$.
Now one can go through vertices in order $\pi$ and color each considered vertex into the minimal color which does not create a monochromatic edge at this moment.
By the construction, the result is a proper coloring $f$ in some number of colors.

Let an \textit{$r$-chain} be a set of edges $g_1,\dots, g_r$ such that $|g_i \cap g_j| = 1$ if $|i-j| = 1$, and 0 otherwise.
We say that an $r$-chain is \textit{ordered in $\pi$} if $i < j$ implies that every vertex of $g_i$ is no bigger that every vertex of $g_j$ with respect to $\pi$.

Pluh{\'a}r noted that the number of colors in $f$ is the maximal size of an ordered chain plus one.
In particular, if $f$ uses more than $r$ colors, then $G$ contains an $r$-chain (in fact, a lot of $r$-chains).
Then Gy{\'a}rf{\'a}s and Lehel~\cite{gyarfas2011trees} extended the latter statement as follows.
If a greedy coloring of a 3-graph $G$ uses more than $r$ colors, then $G$ contains a copy of every hypertree (connected hyperforest) with $r$ edges.

\subsection{Correctness} 
\label{subsec:corr}
We have claimed earlier that the function $m_H(r)$ is properly defined for all $H$ except hyperforests.
Here we provide a brief proof of this statement. Indeed, if $H$ is a hyperforest, then, as was discussed in the previous subsection, any $H$-free hypergraph can be properly colored in $|E(H)|$ colors using greedy coloring method, so $m_H(r)$ is not defined for $r > |E(H)|$.

Otherwise, $H$ contains either a pair of edges intersecting by two vertices, or a \textit{loose cycle} $C_l$ for some $l \geq 3$, i.e. a 3-graph on $2l$ vertices with the set of edges $\{\{1, 2, 3\}, \{3, 4, 5\}, \dots, \{2l - 1, 2l, 1\}\}$.
In the first case any simple 3-graph is $H$-free, so, in view of Theorem~\ref{thm:kostochka-bohman}, $m_H(r) = O(r^4 \log^2 r)$. 
For a graph with a loose cycle $C_l$ one may count the balance of the cycle as follows
\[
\rho(C_l) = \frac{l - 1}{2l - 3} = \frac{1}{2} + \frac{1}{4l - 6}.
\]
The maximum in the definition of the balance is achieved on the whole cycle, so $C_l$ is balanced with the balance greater than $\frac{1}{2}$.
Thus one can apply Theorem~\ref{thm:bohman2010-4}
to obtain the inequality
\[
m_H (r) \leq m_{C_l} (r) = O \left ( (r^2 \log r)^{\frac{3\rho-1}{2\rho-1}} \right ) = O(r^{2l}\log^l r).
\]

%% file: results/polyRamsey.tex
\section{Proofs}
\label{proofs}

%
%
\subsection{Proof of Theorem~\ref{thm:rPoly}}

\subsubsection{The upper bound}
    The lower bound on $R(H, K_t)$ states that for each $t\geq 1$ there is an $H$-free 3-graph $G$ with $|V(G)| = At^{a}$ and the independence number $\alpha(G) < t$. The chromatic number of $G$ satisfies
    \begin{equation}
        \chi(G) \geq
        \frac{|V(G)|}{\alpha(G)}
        \geq At^{a-1}.
    \end{equation}
    Now for a given $r$ we choose integer $t$ such that $A(t-1)^{a-1} \leq r < At^{a-1}$. Then
    \begin{equation}
        |E(G)| \leq \frac{|V(G)|^{3}}6 = \frac{A^3t^{3a}}{6} 
        = \frac{A^{3-\frac{3a}{a-1}}}6 
        (At^{a-1})^{\frac{3a}{a-1}}
        <
        \frac{A^{3-\frac{3a}{a-1}}}6 
        \biggl(A\Bigl(\bigl(1+o(1)\bigr)(t-1)\Bigr)^{a-1}\biggr)^{\frac{3a}{a-1}}
        \leq
        \frac{(1+o(1))^{3a}}{6A^{\frac 3{a-1}}} \cdot r^{\frac{3a}{a-1}}.
    \end{equation}

\subsubsection{The lower bound}
    Suppose that $G$ is an $H$-free 3-graph with $|E(G)| < B'r^{\frac b{b-1}+2}$ that is not $r$-colorable. Let such $r$ be minimal possible. 
    From Lemma~\ref{lm:vsmall}
    it follows that $G[V_{big}]$ cannot be colored with no more than $r/2$ colors.
    Then 
    \begin{equation}\label{eq:egv_big_is_big}
        |E(G[V_{big}])| \geq \frac{B'}{2^{\frac b{b-1}+2}} \cdot r^{\frac b{b-1}+2},
    \end{equation}
    
    Let $V_{big} = V_0 \cup V_1 \cup V_2 \cup \dots$, where
    \begin{equation}
        V_k = \{ v \in V(G) \mid 2^k cr^2 \leq \deg v < 2^{k+1}cr^2\},
    \end{equation}
    and $c = (12e)^{-1}$ is from the definition of $V_{big}$.
    Suppose there is some $k$ for which
    \begin{equation}\label{eq:alphaGVk}
        \alpha(G[V_k]) \geq \frac{B'}{2^{k-1}cr^2}\bigl(r^{\frac b{b-1} + 2} - (r-1)^{\frac b{b-1}+2}\bigr).
    \end{equation}
    Then we can paint the vertices from a corresponding independent set in one color. Now if we remove colored vertices from $G$, we get a 3-graph with less than 
    \begin{equation}
        B'r^{\frac b{b-1}+2} - \frac 12 \cdot 2^{k}cr^2\cdot \alpha(G[V_k]) = B'(r-1)^{\frac b{b-1}+2}
    \end{equation}
    edges that is not $(r-1)$-colorable. That is in contradiction to the minimality of $r$.
    So we can assume that for every $k$ the contrary of~\eqref{eq:alphaGVk} holds, so, by upper bound from~\eqref{eq:ramsey_bounds}, we have
    \begin{equation}
        |V_k| < B\left(
            \frac{B'}{2^{k-1}cr^2}\bigl(r^{\frac b{b-1} + 2} - (r-1)^{\frac b{b-1}+2}\bigr)
                \right)^b.
    \end{equation}
    Using the fact that $x^q - (x-1)^q \leq qx^{q-1}$ for $x,q \geq 1$, we get
    \begin{equation}
        |V_k| < B\left(
            \frac{B'}{2^{k-1}cr^2}\cdot \biggl(\frac{b}{b-1}+2\biggr) \cdot r^{\frac b{b-1} + 1}
                \right)^b
                =
                C \left(\frac{r^{\frac b{b-1}+1}}{2^{k-1}r^2}\right)^b
                =
                \frac{1}{2^{b(k-1)}}\cdot C r^{\frac{b}{b-1}},
    \end{equation}
    where
    \begin{equation}
        C = \frac{B{B'}^{b}\bigl(\frac b{b-1}+2\bigr)^b}{c^b}.
    \end{equation}
    Now let $g_k$ be the number of edges that pass through vertices of $V_k$. We have
    \begin{equation}
        |E(G[V_{big}])| < \sum_{k \geq 0} g_k
        < \sum_{k \geq 0} 2^{k+1}cr^2 \cdot 
                \frac{1}{2^{b(k-1)}}\cdot C r^{\frac{b}{b-1}}
            =
            cC2^{b+1}r^{\frac b{b-1}+2} \sum_{k\geq 0} 2^{(1-b)k}
            <
            \frac{cC2^{b+1}}{1-2^{1-b}} \cdot r^{\frac b{b-1}+2}.
    \end{equation}
    Combining this with~\eqref{eq:egv_big_is_big} and substituting $C$ and $c$ we get
    \begin{equation}
        \frac{B'}{2^{\frac b{b-1}+2}} < \frac{(1/12e)B{B'}^b\bigl(\frac b{b-1}+2\bigr)^b 2^{b+1}}{(1/12e)^b(1-2^{1-b})} 
    \end{equation}
    or
    \begin{equation}
        {B'}^{b-1} > \frac{(1-2^{1-b})}{2^{\frac b{b-1}+b+3} (12e)^{b-1}B\bigl(\frac b{b-1}+2\bigr)^b},
    \end{equation}
    which is a contradiction.
\subsection{Proof of Proposition~\ref{prop:good_upper_bound}}

A direct application of Theorem~\ref{thm:rPoly} gives the upper bound $r^{9+o(1)}$. 
To obtain a better upper bound, let us count the number of edges in the example constructed in~\cite{kostochka2013hypergraph}.
The construction is as follows: let $G_q$ be a generalized quadrangle, i.e. a $(q + 1)$-regular $(q + 1)$-uniform simple $C_3$-free hypergraph on $q^3 + q^2 + q + 1$ vertices.
It is known that such hypergraph exists whenever $q$ is a power of a prime number.
Each edge of $G_q$ is then replaced by a random 3-graph $F_q$ which is built as follows: vertices of $F_q$ are randomly split into a set $V = \{v_{ij} \ |\ 1 \leq i,j\leq \tau\}$ and $2\tau$ sets $S_1, \dots, S_{\tau}, T_1, \dots, T_{\tau}$ of roughly the same size, where $\tau = \lfloor 8 \sqrt{\log q} \rfloor$.
Edges of $F_q$ are all sets of the form $\{v_{ij}, a, b\}$, where $a \in S_i$, $b \in T_j$.

The result of the construction is a hypergraph on $cq^3$ vertices.
With a high probability its independence number does not exceed
\[
\frac{2\tau n}{q}=16q^2\sqrt{\log q},
\]
i.e. its chromatic number is at least $cq/\sqrt{\log q}$.

At the same time, this hypergraph has
\[
cq^3\tau^2\left(\frac{q+1-\tau^2}{2\tau}\right)^2=cq^5
\]
edges.
Substituting $cq/\sqrt{\log q}=r$, we get
\[
m_{C_3}(r) \leq cr^5(\log r)^{5/2}.
\]




%% file: results/tournament_K34.tex
\subsection{Proof of Theorem~\ref{thm:rExp}}

\subsubsection{The upper bound}

For any $t\geq 1$ there is a $H$-free 3-graph $G_t$ with $|V(G_t)| = t^{At}$ and $\alpha(G_t) < t$.
Then
\[
\chi(G_t) \geq \frac{|V(G_t)|}{\alpha(G_t)} > \frac{t^{At}}{t} = t^{At-1},
\]
so one may put $r = t^{At-1}$. First we prove the desired bound only for such $r$. Hence
\[
\ln r = (At-1)\ln t,
\]
so
\[
\frac{1}{A} \frac{\ln r}{\ln \ln r} = \frac{1}{A} \frac{(At-1)\ln t}{\ln((At-1)\ln t)} = (1+o(1))t.
\]
Finally,
\[
|E(G_t)| \leq |V(G_t)|^3 = t^{3At} = r^3t^3 = (1+o(1)) \left( \frac{r \ln r}{A \ln \ln r}\right)^3.
\]

Then, consider any $r$ between $(t-1)^{A(t-1)-1}$ and $t^{At-1}$. Let $A'$ be such that $r = t^{A't-1}$.
For a large enough $t$ one has
\[
t^{At/2-1} < (t-1)^{A(t-1)-1},
\]
so $A > A' > A/2$.
Consider any subgraph $G'_t$ of $G_t$ on $t^{A't}$ vertices. 
It is $H$-free and has the independence number smaller than $t$, so the chromatic number of $G'_t$ is at least $r$.
Then one may repeat the previous argument to show
\[
|E(G'_t)| \leq |V(G'_t)|^3 = t^{3A't} = r^3t^3 = (1+o(1)) \left( \frac{r \ln r}{A' \ln \ln r}\right)^3 < (1+o(1)) \left( \frac{2 r \ln r}{A \ln \ln r}\right)^3.
\]

\subsubsection{The lower bound}

Essentially, this proof follows along the proof of the lower bound in Theorem~\ref{thm:rPoly}.
First, we show that for each $\eps \in (0, 1 / b)$ there is some constant $C_\eps > 0$ such that
    \begin{equation}\label{eq:B_eps}
        m_H(r) \geq C_\eps r^3 \ln^{1/b-\eps}r.
    \end{equation}
    Suppose that $G$ is an $H$-free 3-graph with $|E(G)| < C_\eps r^3\ln^{1/b-\eps}r$ that is not $r$-colorable. Here $C_\eps$ is some constant which depends only on $\eps$ and which will be chosen later.
    Let such $r$ be minimal possible. 
    By Lemma~\ref{lm:vsmall} we have that $G[V_{big}]$ should not be $\lfloor r/2\rfloor$-colorable.
    
    Let us look at the value of $\alpha(G[V_{big}])$.
    If 
    \begin{equation}\label{eq:alphagv_big_is_big}
        \alpha(G[V_{big}]) \geq
        \frac{2C_\eps}{cr^2}\bigl(
            r^{3}\ln^{1/b-\eps}r - (r-1)^3\ln^{1/b-\eps}(r-1)
        \bigr),
    \end{equation}
    then we can paint vertices of the biggest independent set of $G[V_{big}]$ in one color.
    By removing colored vertices from $G$ we get $H$-free 3-graph that is not $(r-1)$-colorable but has less than
    \begin{equation}
        C_\eps r^3\ln^{1/b-\eps}r - \frac {cr^2}2 \cdot \frac{2C_\eps}{cr^2}\bigl(
            r^{3}\ln^{1/b-\eps}r - (r-1)^3\ln^{1/b-\eps}(r-1)
        \bigr)
        =
        C_\eps(r-1)^3\ln^{1/b-\eps}(r-1)
    \end{equation}
    edges, which contradicts the minimality of $r$.
    
    So the contrary to~\eqref{eq:alphagv_big_is_big} must hold.
    Since $r^3\ln^a r - (r-1)^3\ln^a(r-1) < 4r^2\ln^a r$ for all $r\geq 1$ and $a \geq 0$, we have
    \begin{equation}
        \alpha(G[V_{big}]) <
        \frac{8C_\eps}{c} \cdot \ln^{1/b-\eps}r.
    \end{equation}
    Due to the bound~\eqref{eq:ramsey_exp_b},
    \begin{equation}
        \left |V_{big} \right| < \left(\frac{8C_\eps}{c} \cdot \ln^{1/b-\eps}r\right)^{B((8C_\eps \ln^{1/b-\eps}r)/ c)^b}.
    \end{equation}
    We take logarithm of the right-hand side and see that when $C_\eps \leq c/8$
    \begin{equation}
        B\left(\frac{8C_\eps}c \ln^{1/b-\eps} r\right)^b \cdot \left(\ln\frac{8C_\eps}{c} + \ln \ln^{1/b-\eps}r\right)
        \leq
        B\left(\frac{8C_\eps}c\right)^b \ln^{1-b\eps} r \cdot \ln \ln^{1/b-\eps}r
        <
        B\left(\frac{8C_\eps}c\right)^b\ln^{1-b\eps} r \cdot \ln \ln^{1/b}r.
    \end{equation}
    There exists $x_\eps > 0$ such that the function $\ln^{b\eps} x / \ln\ln^{1/b} x$ is positive and strictly increases when $x \geq x_\eps$.
    Then for any $C_\eps$ such that
    \begin{equation}
        \left(\frac{8C_\eps}{c} \cdot \ln^{1/b-\eps}x\right)^{B((8C_\eps \ln^{1/b-\eps}x)/ c)^b} < x \text{ for any } x < x_\eps \text{ and }
        C_\eps^b < \frac 1B \cdot \left(\frac c{8}\right)^b \frac{\ln^{b\eps} x_\eps}{\ln \ln^{1/b} x_\eps},
    \end{equation}
    we either have $|V_{big}| < r$ if $r < x_\eps$, or we have
    \begin{equation}
        \ln |V_{big}| <
        BC_\eps^b\left(\frac{8}c\right)^b \ln^{1-b\eps} r \cdot \ln \ln^{1/b} r
        < \ln^{1-b\eps} r \cdot \ln^{b\eps} r
        = \ln r
    \end{equation}
    if $r \geq x_\eps$, which also means that $|V_{big}| < r$. So vertices of $V_{big}$ can be colored in $\lfloor r/2 \rfloor$ colors with no more than two vertices of same color. It means that there is a proper coloring of $G$ in $r$ colors. Obtained contradiction finishes the proof of~\eqref{eq:B_eps}.
    
    It remains to construct some function $g$ such that $g(r) \to 0$ as $r \to \infty$ and
    \begin{equation}
        m_H(r) \geq r^3 \ln^{1/b-g(r)} r.
    \end{equation}
    For each $n$ there is some $r_n$ such that for $r \geq r_n$
    \begin{equation}
        C_{1/n} \ln^{1/n} r > 1,
    \end{equation}
    which means that
    \begin{equation}
        m_H(r) \geq r^3 \ln^{1/b-2/n} r.
    \end{equation}
    Finally, for $r \geq r_1$ let
    \begin{equation}
        g(r) = \frac 2{\max\limits_{\substack{n \leq r\\ r_n \leq r}} n}.
    \end{equation}
    Clearly, $g(r)$ is non-increasing and can be arbitrary small, so $g(r) \xrightarrow[r\to\infty]{}0$.

    
    



%% file: results/pluhar.tex
\subsection{Proof of Proposition~\ref{prop:e288}}


Suppose the contrary, i.\,e.\ there is an $H$-free 3-graph $G$ with the chromatic number more than $r$ and the number of edges at most $Cr^4$. Let such $r$ be minimal possible.

Let us try to color $G[V_{big}]$ in $r'=\lfloor r/2 \rfloor$ colors greedily with respect to an arbitrary linear order $\pi$ on the vertex set. 
If it is impossible then there is a vertex $v$ such that it is the last 
vertex in edges $f_1,\dots,f_{r'}$ such that 
set $f_i \setminus \{v\}$ has color $i$. Then an arbitrary choice of one vertex from each set $f_i \setminus \{v\}$ leads to an independent set $I$ of size $r'$, since $G$ is $H$-free.

By removing $I$ from $G$ we get some 3-graph that is not $(r-1)$-colorable, so it must have more than $C(r-1)^4$ edges.
So we removed less than
\begin{equation}
    Cr^4 - C(r-1)^4 < C\cdot 4r^3
    = \frac{er^3}{72}
\end{equation}
edges. But the sum of degrees of vertices in $I$ (in the initial graph $G$) is at least $r'\cdot cr^2$, so it intersects at least
\begin{equation}
    \frac{cr^2r'}2 \geq \frac{cr^3}6 = \frac{er^3}{72}
\end{equation}
edges. That is a contradiction.

%% file: results/rb.tex
\subsection{Proof of Theorem~\ref{thm:turan}}



Let $C_2 > 0$.
Consider a 3-graph $G$ with
\begin{equation}\label{eq:r4t5}
    |E(G)| \leq \frac{C_2r^{\frac{2b}{b-1}}}{(\ln \ln r)^{1+\frac{2b}{b-1}}}.
\end{equation} 
We will show that it is $r$-colorable if $r$ is sufficiently big.
Let $V_0 = V(G)$ and
\begin{equation}
    V_i = 
    \left\{v \in V_{i-1}: 
    \deg_{G[V_{i-1}]} v \geq \frac1{C_3}\left(\frac{r}{\ln \ln r} \right )^2 \right\},
    \quad
    i \geq 1,
\end{equation}
where $C_3 > 0$ is chosen to satisfy the inequality
\begin{equation}
    \frac 1{C_3}
    \left(
        \frac{r}{\ln\ln r}
    \right)^2
    \leq
    \frac 1{3e}
    \left[
        \frac{r}{\log_b\ln r}
    \right]^2
\end{equation}
for, let's say, $r > 100$,
so, by Lemma~\ref{lm:3graph-coloring}, every 3-graph $G[V_i \setminus V_{i+1}]$ for $i\geq 0$ has a proper coloring in $\bigl[\frac{r}{\log_b\ln r}\bigr]$ colors. Clearly, $C_3$ can be chosen independently of $r$.

Now we are going to estimate sizes of $E(G[V_i])$ and $V_{i+1}$ from above for each $i \geq 0$.
As $V_{i+1} \subset V_i$, we have
\begin{equation}
    \left | E(G[V_i]) \right | \geq
    \frac 13 \sum_{v \in V_{i+1}} \deg_{G[V_i]} v
    \geq 
    \frac 13 \cdot \left | V_{i+1} \right | \cdot 
        \frac 1{C_3}
        \left(
            \frac{r}{\ln \ln r}
        \right)^2,
\end{equation}
so
\begin{equation}
    \label{eq:v_bound}
    \left| V_{i+1} \right | \leq 3C_3 
        \left(
            \frac{\ln \ln r}{r}
        \right)^2
        \left | E(G[V_i]) \right |.
\end{equation}
Also, we have bound~\eqref{eq:r4t5} for $\left | E(G[V_0]) \right | = |E(G)|$, and, by~\eqref{eq:exnH},
\begin{equation}
    \label{eq:e_bound}
    \left | E(G[V_i]) \right | \leq C_1 (\left | V_i \right |)^b \ \text{for $i \geq 1$}.
\end{equation}
Combining~\eqref{eq:e_bound} and~\eqref{eq:v_bound} and using induction on $i$, we obtain for $i\geq0$ 
\begin{equation}\label{eq:egvi_bound}
    \left | E(G[V_{i}]) \right | \leq 
        C_1^{\frac{b^i-1}{b-1}}
        C_2^{b^i}
        (3C_3)^{\frac{b^{i+1}-b}{b-1}}
        \cdot
        \frac{r^{\frac{2b}{b-1}}}
            {(\ln \ln r)^{b^{i}+\frac{2b}{b-1}}},
\end{equation}
\begin{equation}\label{eq:vi_bound}
    \left | V_{i+1} \right | \leq 
        C_1^{\frac{b^i-1}{b-1}}
        C_2^{b^i}
        (3C_3)^{\frac{b^{i+1}-1}{b-1}}        
        \cdot
         \frac{r^\frac{2}{b-1}}{(\ln \ln r)^{b^i + \frac{2}{b-1}}}.
\end{equation}
The last inequality implies that $V_N$ is empty for $N$ and $r$ large enough.
Namely, let $C = C_1^{\frac 1{b-1}}C_2(3C_3)^{\frac b{b-1}}$. Then for $r > r_0(C_1,C_3)$ from last inequality we get
\begin{equation}\label{eq:V_N}
    \left | V_{i+1} \right |
    \leq
        \frac{r^{\frac 2{b-1}}}
        {C_1^{\frac 1{b-1}} (3C_3)^{\frac 1{b-1}} \cdot (\ln \ln r)^{\frac{2b}{b-1}}}
        \bigg/
        \left(
            \frac{\ln\ln r}
                {C}
        \right)^{b^{i}}
    \leq
        r^{\frac 2{b-1}}
        \bigg /
        \left(
            \frac{\ln\ln r}C
        \right)^{b^{i}}.
\end{equation}
So for $N > \log_b \ln r$ 
\begin{equation}
    |V_{i+1}|
    \leq 
    r^{\frac 2{b-1}}
    \bigg /
    \left(
        \frac{\ln\ln r}C
    \right)^{\ln r}
    =
    \left(
        \frac{C e^{\frac 2{b-1}}}
            {\ln \ln r}
    \right)^{\ln r},
\end{equation}
which is less than 1 when $\ln \ln r > Ce^{2/(b-1)}$, i.e.
\begin{equation}
    r > e^{e^{Ce^{2/(b-1)}}}.
\end{equation}
We have divided $V$ into at most $\log_b \ln r$ sets of form $V_i\setminus V_{i+1}$, and each 3-graph $G[V_i\setminus V_{i+1}]$ has a proper coloring with at most $\frac r{\log_b\ln r}$ colors, so $G$ is $r$-colorable. Therefore, we can put $R = \max \left (100, r_0, e^{e^{Ce^{2/(b-1)}}} \right)$.

\subsection{Proof of Proposition~\ref{prop:edges_ordering}}
Suppose that a 3-graph $G$ has more than $(t-3)\binom n2$ edges. 
We will show that some subgraph of $G$ is isomorphic to $H$.

Let us repeat the following procedure: if there is an edge $g = \{v_1,v_2,v_3\}$ such that the pair $(v_1,v_2)$ is contained in at most $t-3$ edges (including $g$), remove $g$.
There will be removed at most $(t-3)\binom{n}{2}$ edges, and after that each pair of vertices will be joined by either 0 or at least $t-2$ edges. Denote resulting 3-graph by $G'$.

Let $H_i$ be a subgraph of $H$ induced by edges $h_1,\dots,h_i$.
We have removed at most $(t-3)\binom{n}{2}$ edges from $G$, so $\left | E(G') \right | \geq 1$. A subgraph of $G'$ induced by one edge is isomorphic to $H_1$. Now we will show that if there is a subgraph $H'_{i - 1} \subset G'$ isomorphic to $H_{i-1}$, then there is a subgraph isomorphic to $H_{i}$. All we have to do is find an edge which can be put in correspondence to $h_i$.
$H_{i-1}$ has an edge $h_{j(i)} \supset \{a_i, b_i\}$. Consider $v,u \in V(H'_{i - 1})$ that correspond to $a_i$ and $b_i$. There is an edge in $H'_{i - 1}$ that corresponds to $h_{j(i)}$, so $v$ and $u$ are joined by at least $t-2$ edges in $G'$. We have $\left | V(H'_{i - 1}) \right | - 2 = \left | V(H_{i-1}) \right | - 2 = i - 1 < t - 2$, so there exists $e \in E(G')$ such that $e = \{v, u, w\}$ and $w \notin V(H'_{i - 1})$.
This edge $e$ can be put in correspondence to $h_i$.

So there is a subgraph of $G'$ that is isomorphic to $H_{t-2} = H$ and, since $G'\subset G$, it is a subgraph of $G$.



%% file: results/diss.tex
\section{Discussion}
\label{diss}

\paragraph{Questions from~\cite{bohman2010coloring}.}

\begin{conjecture}
There exists a simple 3-graph $H$ for which 
\[
m_H(r) = r^{3+o(1)}.
\]
\end{conjecture}

\begin{problem}
Characterize all 3-graphs $H$ such that 
\[
m_H(r) = r^{3+o(1)}.
\]
\end{problem}

\paragraph{Regularity issues.}
It is known~\cite{cherkashin2020regular} that sequence $\frac{m(r)}{r^3}$ has a positive limit.
The same statement for quantity $m_H(r)$ is of interest ($H$ is assumed to be fixed).
In particular, it is unknown even for the case of simple 3-graphs, i.e. when $H = \{\{1,2,3\}, \{1,2,4\}\}$, for which the asymptotics is known: 
\[
m_H(r) = \Theta(r^4\log^2 r).
\]